\documentclass[12pt]{article}
\usepackage{amsmath}
\usepackage{amsfonts}

\newcommand{\prob}[1]{\mathbb{P}\left( #1 \right)}

\newcommand{\E}{\mathbb{E}}
\newcommand{\Var}{{\rm Var}}
\newcommand{\Cov}{{\rm Cov}}
\newcommand{\Abs}[1]{\left\vert #1 \right\vert}
\newcommand{\norm}[1]{\Vert #1\Vert_L}
\newcommand{\clt}{Central Limit Theorem}
\newcommand{\slln}{Strong Law of Large Numbers}
\newcommand{\sllns}{Strong Laws of Large Numbers}

\newtheorem{theorem}{Theorem}

\newtheorem{definition}[theorem]{Definition}
\newtheorem{example}[theorem]{Example}

\newtheorem{lemma}[theorem]{Lemma}

\newtheorem{remark}{Remark}

\newenvironment{proof}[1][Proof]{\noindent\textbf{#1.} }{\ \rule{0.5em}{0.5em}}
\begin{document}

\title{Asymptotic Results for Certain Weak Dependent Variables}
\author{Idir ARAB and Paulo Eduardo Oliveira \\
CMUC, Department of Mathematics, University of Coimbra, Portugal\\
e-mail: idir@mat.uc.pt\\
paulo@mat.uc.pt}

\date{ }
\maketitle

\begin{abstract}
We consider a special class of weak dependent random variables with control on covariances of Lipschitz transformations. This class includes, but is not limited to, positively, negatively associated variables and a few other classes of weakly dependent structures. We prove a \slln\ with a characterization of convergence rates which is almost optimal, in the sense that it is arbitrarily close to the optimal rate for independent variables. Moreover, we prove an inequality comparing the joint distributions with the product distributions of the margins, similar to the well known Newman's inequality for characteristic functions of associated variables. As a consequence, we prove a \clt\ together with its functional counterpart, and also the convergence of the empirical process for this class of weak dependent variables.\\


\end{abstract}
\ \ *This work was partially supported by the Centre for Mathematics of the University of Coimbra -- UID/MAT/00324/2013, funded by the Portuguese Government through FCT/MEC and co-funded by the European Regional Development Fund through the Partnership Agreement PT2020.
\section{Introduction}
Limit theorems, either with respect to almost sure convergence or convergence in distribution are a central subject in statistics. In more recent years, many authors were interested on the asymptotics for dependent sequences of variables. Several forms of controlling the dependence have been proposed, many of them describing a control on covariances of transformations of variables.
Mostly, this control may be thought of as measuring the degree of dependence  between a past and a sufficiently separated future.
These dependence structures are commonly named weak dependence,
and are described using specific families of transformations of the random variables. We refer the reader to Doukhan and Louhichi~\cite{DL99} or Dedecker et al.~\cite{Ded07} for some examples and relations between such dependence notions.
Many of these notions stemmed from the positive dependence and association introduced by Lehmann~\cite{Leh66} and Esary, Proschan and Walkup~\cite{EPW67}, respectively. Association was the first of these two notions to attract the interest of researchers, and as expected, \sllns\ and \clt s were eventually proved. We refer the reader to the monographs by Bulinski and Shashkin~\cite{BS07}, Oliveira~\cite{Oli12} or Prakasa Rao~\cite{PR12} for an account of relevant literature. Inevitably, several variations and extensions of these dependence notions were introduced and limit theorems were established. Among these, the negative association defined by Joag-Dev and Proschan~\cite{JP83} was one of the most popular, with various different extensions introduced in more recent years: extended negative dependent (END) introduced by Liu~\cite{L09}, widely orthant dependent (WOD) introduced by Wang, Wang and Gao~\cite{WWG13} among other variations.
We will be interested in a particular version of weak dependence defined in the same spirit as in Doukhan and Louhichi~\cite{DL99}, instead of a direct variation on the inequalities that express the positive or the negative dependence.

The proof techniques for the dependence structures mentioned rely essentially on an adequate control of the covariances between appropriate families of transformations of the random variables. Thus, it was natural to define the dependence control through some upper bound of a convenient family of covariances, usually characterized by a suitable family of transformations of the variables, leading to the weak dependence notions, as introduced by Doukhan and Louhichi~\cite{DL99}.
Different dependence notions are defined by considering distinct families of transformations.
For an account on some fo these dependence structures and their relations, we refer the reader to the monograph by Dedecker et al.~\cite{Ded07}.

In this paper, we will be interested in a particular version of weak dependence, somewhat similar to the quasi-association as introduced in Bulinski and Suquet~\cite{BulSuq01}, that includes the positive, negative dependence notions referred above and the quasi-association.
We will also provide nontrivial examples showing that the inclusions between these classes of dependent variables is strict.
For the weak dependence notion we are defining, we will prove a \slln, with a characterization of rates for both bounded and unbounded random variables, a \clt\ and an invariance principle. 
We should compare the results proved here with the ones already available in the literature for the various dependence structures. As what concerns convergence in distribution, our results are of similar strength, essentially only providing a unified approach to the different frameworks. For the almost sure convergence, the assumptions and the derived rates are again similar to most of the known results for negatively or positively associated variables.
However,
for weak dependent families of variables,
the only inequality controlling tail probabilities (see Corollary~1 in \cite{DL99}) is a Bernstein type inequality, that has a relatively weak form. Later, Corollary~4.1 and Theorem~4.5 in \cite{Ded07} and Kallabis and Neumann~\cite{KN06} also prove exponential inequalities that are analogous to the Bernstein inequalities, but again with weaker exponents in their upper bounds. This means that although \sllns\ may be derived, not only the assumptions will become stronger, but convergence rates that follow will not be almost optimal, in the sense that these rates may be arbitrarily close to the well known rates for independent variables. In the present paper, the version of weak dependence we will be studying allows for the adaptation of techniques used for associated variables (see, for example, Ioannides and Roussas~\cite{IR98}, Oliveira~\cite{Oli05}, Sung~\cite{Sung07}) providing stronger forms of the Bernstein-type inequality, meaning that we will obtain almost optimal convergence rates. 

The paper is organized as follows: Section~2 defines the framework, Section~3 proves some basic inequalities needed for the control of the almost sure convergence, which is the object of Section~4, where \sllns\ for bounded and unbounded random variables, with characterization of rates, are proved.  Finally, in Section~5, we extend the Newman inequality for characteristic functions to the present dependence structure, from which a \clt, an invariance principle and the convergence of the empirical process follow.

\section{Definitions and framework}

Let $X_n$, $n\geq 1$, be centered random variables and define $S_n=X_1+\cdots+X_n$. As mentioned before, we will be interested in a particular form of weak dependence, according to the following definition.
\begin{definition}
\label{def:weak}
The random variables $X_i$, $i=1,\ldots,n$, is said to be L-weakly dependent if there exist nonnegative coefficients $\gamma_k$, $k\geq 1$, such that for every disjoint subsets $I,J\subset\{1,\ldots,n\}$ and real valued Lipschitz functions $f$ and $g$, defined on the appropriate Euclidean spaces, 
the following inequality is satisfied:
$$
\Abs{\Cov\left( f\left( X_{i},i\in I\right), g\left( X_{j},j\in J\right) \right)}
\leq \norm{f}\norm{g}\sum_{i\in I}\sum_{j\in J}\gamma_{\Abs{j-i}},
$$
where $\norm{f}$ represents the Lipschitz norm of $f$:
$$
\norm{f}=\sup_{x\ne y}\frac{\Abs{f(x)-f(y)}}{\Abs{x-y}}.
$$
An infinite family of random variables is said to be L-weakly dependent if every finite subfamily is L-weakly dependent and the coefficients define a convergent series. 
\end{definition}
This is a form of weak dependence in the same spirit as in Doukhan and Louhichi~\cite{DL99} or Dedecker et al.~\cite{Ded07}. With respect to the discussion in \cite{Ded07}, this dependence follows from what these authors called the $\kappa$ or the $\zeta$ coefficients. This means the examples of L-weakly dependent sequences include positively associated, negatively associated, Gaussian sequences or models for interacting particles systems (see Section~3.5.3 in \cite{Ded07} for details for this last example). Moreover, the notion of quasi-association, introduced by Bulinski and Suquet~\cite{BulSuq01}, is also included in the L-weak dependence structure by choosing $\gamma_k=\Cov(X_1,X_{k+1})$, of course assuming the stationarity of the random variables.
The inclusion between these families of dependent variables is strict, as we will be showing by exhibiting a few examples.
\begin{example}
Let $\xi_n$, $n\in\mathbb{Z}$, be a sequence of independent random variables with variances $\sigma_n^2$. Given $p\geq 1$ and $\alpha_1,\ldots,\alpha_p\in\mathbb{R}$, define, for each $n\geq 1$, $X_n=\sum_{j=1}^{p}\alpha_j \xi_{n-j}$. It is well known that the sequence $X_n$ is positively associated if and only if the $\alpha_i$ all have the same sign. Now, if we choose the coefficients $\alpha_1$ and $\alpha_p$ positive, and $\alpha_2$, and $\alpha_{p-1}$ negative, it follows that
$$
\begin{array}{l}
\Cov(X_n,X_{n+p-1})=\alpha_1\alpha_p\sigma_{n-1}^2>0, \\
\Cov(X_n,X_{n+p-2})=\alpha_1\alpha_{p-1}\sigma_{n-1}^2+\alpha_2\alpha_{p}\sigma_{n-2}^2<0.
\end{array}
$$
Hence, the sequence $X_n$, $n\geq 1$, is neither negatively associated nor positively associated. However, it is easily verified that it is quasi-associated.
\end{example}
\begin{remark}
As composition of Lipshcitz functions is still Lipschitzian, and quasi-associated variables are L-weak dependent, it follows that Lipschitz transformations of quasi-associated variables are L-weak dependent. However, as shown by the following example, the transformed variables are not necessarily quasi-associated.
\end{remark}
\begin{example}
\label{ex:LL}
Let $\xi_n$, $n\geq 1$, be a sequence of independent and identically distributed random variables, $\alpha_n$, $n\geq 1$, a sequence of real numbers, and define, for each $n\geq 1$, $X_n=\sum_{i=1}^n\alpha_i\xi_i$. Taking all the coefficients positive, this sequence is positively associated, therefore, also quasi-associated. Consider a Lipschitz function $g$ such that $g(x+y)=g(x)g(y)$ and denote $f=g^{-1}$, that is assumed to be also Lipschitzian. Finally, define $Y_n=g(X_n)$. It is now easily verified that $\Cov(X_1,X_2)=\alpha_1^2\Var(\xi_1)$, while
$
\Cov(Y_1,Y_2)=\E(g(\alpha_2\xi_2))\Var(g(\alpha_1\xi_1)).
$
If we choose the common distribution of the $\xi_n$ and the function $g$ such that $\lim_{\alpha_1\rightarrow\infty}\Var(g(\alpha_1\xi_1))=0$, it follows that the inequality
\begin{equation}\label{ineq_Lip}
  \Cov(X_1,X_2)=\Cov(f(g(X_1)),f(g(X_2)))\leq \norm{f}^2\Cov(g(X_1),g(X_2))
\end{equation}
cannot be fulfilled, at least for $\alpha_1$ large enough. Therefore the random variables $Y_n=g(X_n)$, $n\geq 1$, cannot be quasi-associated.
\end{example}
%
\begin{example}
A concrete example may be obtained taking $g(x)=e^{-x}$ and the $\xi_n$ uniform on some closed interval. Note that, although $g(x)$ and $f(x)=g^{-1}(x)=-\log x$ are not Lipschtizian in all their domain, they are Lipschitz in the support of the variables to which we will be applying and, as we will be computing expectations, this is enough to characterize the L-weak dependence. The uniform distribution is just an easily verifiable example. Other distributions may be considered. In fact, representing by $M_\xi$ the moment generating function of the initial random variables $\xi_n$, we have that $\Var(g(\alpha_1\xi_1))=M_\xi(-2\alpha_1)-M_\xi^2(\alpha_1)$ and this converges to 0 under the assumption $\lim_{\alpha\rightarrow\infty}M_\xi(\alpha)=0$. Besides, to have the Lipschitizianity of the transformations considered, the $\xi_n$ should have a compact support.
\end{example}
\begin{example}
Another concrete construction based on Example~\ref{ex:LL} may be obtained choosing $g(x)=e^{-x^2/\beta}+x$, where $\beta >0$. This a strictly increasing and Lipschitz function in the whole real line. Although it is invertible, we do not have an explicit expression for $g^{-1}$. Assume the $\xi_n$, $n\geq 1$, are independent and nonnegative valued, and, as above, take $X_n=\sum_{i=1}^n\alpha_i\xi_i$, where $\alpha_i>0$. This implies that the sequence $X_n$, $n\geq 1$, is associated. Moreover, it is easily seen verified that $\Cov(X_1,X_2)=\alpha_1^2\Var(\xi_1)$.
Consider now
\begin{eqnarray*}
\lefteqn{\Cov(g(X_1),g(X_2))} \\
 & &=\alpha_1^2\Var(\xi_1)+\Cov\left(e^{-\alpha_1^2\xi_1^2}+\alpha_1\xi_1,e^{-(\alpha_1\xi+\alpha_2\xi_2)^2}\right) \\
 & &\qquad
   +\alpha_1\Cov\left(\xi_1,e^{-(\alpha_1\xi+\alpha_2\xi_2)^2}\right).
\end{eqnarray*}
The first covariance above considers an increasing transformation of the $(\xi_1,\xi_2)$ and a decreasing transformation of the same random vector. So the association of the vector implies that this term is negative. The same argument applies to the second covariance. If the random variables $X_1$ and $X_2$ are to be quasi-associated, then (\ref{ineq_Lip}) must be fulfilled. For the present construction, we have
$$\norm{f}=\norm{g^{-1}}=\frac{
1}{1-\sqrt{2/\beta}e^{-1/2}}.
$$
Now, choosing $\beta$ large enough and taking into account the previous comments, (\ref{ineq_Lip}) will not be verified. Hence, for such choice of the parameters, the random variables $X_1$ and $X_2$ cannot be quasi-associated. However, being Lispchitz transformations of independent variables, they are L-weak dependent.
\end{example}

We will be assuming throughout this paper that
\begin{equation}\label{eq:sigma}
\frac{1}{n}\E S_n^2\longrightarrow\sigma^2\in(0,\infty).
\end{equation}
\begin{remark}
This condition follows immediately from the convergence of the series of L-weak dependence coefficients $\gamma_k$ (see Lemma 1.1 in Rio~\cite{Rio2013}).
\end{remark}
This, obviously, implies that for $n$ large enough, we have $\E S_n^2\leq 2\sigma^2 n$. Besides, we will need to decompose $S_n$ into an appropriate sum of blocks. For this purpose, consider an increasing sequence of integers $p_n\leq\frac{n}{2}$ such that $p_n\longrightarrow+\infty$, put $r_{n}=\lfloor\frac{n}{2p_{n}}\rfloor$, where $\lfloor x\rfloor$ represents the integer part of $x$, and define the blocks:
\begin{equation}
\label{decomp1}
Y_{j,n}=\sum_{k=(j-1)p_n+1}^{jp_n}X_k,\quad j=1,\ldots,2r_n.
\end{equation}
Notice that, if the random variables are bounded by $c>0$, then $\Abs{Y_{j,n}}\leq cp_n$. Moreover, define the alternate sums:
$$
Z_{n,od}=\sum_{j=1}^{r_n}Y_{2j-1,n}\qquad\mbox{and}\qquad Z_{n,ev}=\sum_{j=1}^{r_n}Y_{2j,n}.
$$
Note that $S_{n}=Z_{n,od}+Z_{n,ev}+R_{n}$, where
$$
R_{n}=\sum_{j=2r_{n}p_{n}+1}^{n}Y_{j}.
$$
Finally, we introduce the generalized Cox-Grimmett coefficients adapted to the L-weak 
dependence structure,
\begin{equation}
\label{CG}
v(n)=\sum_{k=n}^\infty\gamma_k.
\end{equation}

\section{Inequalities for bounded variables}
\label{ineq}
This section establishes a few inequalities that are the basic tools for proving the almost sure convergence results. The inequalities below are extensions of analogous results for associated random variables. We start by proving a bound for the Laplace transform of the blocks $Y_{j,n}$.
\begin{lemma}
\label{lem2}
Assume that the sequence $X_n$, $n\geq1$, is stationary, there exists some $c>0$ such that for every $n\geq 1$, $\Abs{X_n}\leq c$ almost surely, and that (\ref{eq:sigma}) holds. Let $d_n>1$, $n\geq 1$, be a sequence of real numbers. Then, for every $t\leq \frac{d_n-1}{d_n}\frac{1}{cp_n}$ and $n$ large enough,
$$
\E e^{tY_{j,n}}\leq \exp\left(2t^2\sigma^2 p_n d_n\right).
$$
\end{lemma}
\begin{proof}
Using a Taylor expansion and taking into account the boundedness of the random variables, we have
\begin{eqnarray*}
\lefteqn{\E e^{tY_{j,n}}=1+\sum_{k=2}^\infty\frac{t^k\E Y_{j,n}^k}{k!}} \\
 & &\leq
  1+\sum_{k=2}^\infty\frac{t^kc^{k-2}p_n^{k-2}\E Y_{j,n}^2}{k!}\leq 1+t^2\E Y_{j,n}^2\sum_{k=2}^\infty (tcp_n)^{k-2}.
\end{eqnarray*}
It follows from the assumption on $t$ that $tcp_n\leq\frac{d_n-1}{d_n}<1$, thus, as the sequence $X_n$, $n\geq1$, is stationary, we may write
$$
\E e^{tY_{j,n}}\leq 1+\frac{t^2\E S_{p_n}^2}{1-tcp_n}.
$$
We have $\frac{1}{1-tcp_n}\leq d_n$, so
$\E e^{tY_{j,n}}\leq 1+2t^2\sigma^2 p_n d_n\leq\exp\left(2t^2\sigma^2 p_n d_n\right)$.
\end{proof}


Considering now L-weakly dependent variables, we prove an upper bound for $\E e^{tZ_{n,od}}$.
\begin{lemma}
\label{lem2+1}
Assume the conditions of Lemma~\ref{lem2} are satisfied and the sequence of random variables $X_n$, $n\geq 1$, is L-weakly dependent. Then, for every $t\leq \frac{d_n-1}{d_n}\frac{1}{cp_n}$ and $n$ large enough, we have
\begin{equation}
\label{eq:lem2}
\displaystyle\E e^{tZ_{n,od}}\leq t^2e^{\frac{tcn}{2}}p_n v(p_n)\sum_{j=0}^{r_n-2}\exp\left(jtp_n(2t\sigma^2 d_n-c)\right)+\exp\left(t^2\sigma^2 nd_n\right).
\end{equation}
\end{lemma}
\begin{proof}
Remark first that $\E e^{tZ_{n,od}}=\E\left(\prod_{j=1}^{r_n} e^{tY_{2j-1,n}}\right)$.
Now, by adding and subtracting appropriate terms, we find that
\begin{eqnarray*}
\lefteqn{\E\left(\prod_{j=1}^{r_n} e^{tY_{2j-1,n}}\right)} \\
 & & =
   \Cov\left(\prod_{j=1}^{r_n-1} e^{tY_{2j-1,n}},e^{tY_{2r_n-1,n}}\right)+\E\left(\prod_{j=1}^{r_n-1} e^{tY_{2j-1,n}}\right)\E e^{tY_{2r_n-1,n}}\\
 & &=\Cov\left(\prod_{j=1}^{r_n-1} e^{tY_{2j-1,n}},e^{tY_{2r_n-1,n}}\right) \\
 & &\qquad
          +\Cov\left(\prod_{j=1}^{r_n-2} e^{tY_{2j-1,n}},e^{tY_{2r_n-3,n}}\right)\E e^{tY_{2r_n-1,n}} \\
 & &\qquad
    +\E\left(\prod_{j=1}^{r_n-3} e^{tY_{2j-1,n}}\right)\E e^{tY_{2r_n-3,n}}\E e^{tY_{2r_n-1,n}}.
\end{eqnarray*}
Before iterating this procedure remark that due to the stationarity of the sequence of random variables $X_i$, $\E e^{tY_{2r_n-3,n}}=\E e^{tY_{2r_n-1,n}}=\E e^{tY_{1,n}}$, so the previous expression may be rewritten as
\begin{eqnarray*}
\lefteqn{\E\left(\prod_{j=1}^{r_n} e^{tY_{2j-1,n}}\right)}\\
 & &=
 \Cov\left(\prod_{j=1}^{r_n-1} e^{tY_{2j-1,n}},e^{tY_{2r_n-1,n}}\right)
          +\Cov\left(\prod_{j=1}^{r_n-2} e^{tY_{2j-1,n}},e^{tY_{2r_n-3,n}}\right)\E e^{tY_{1,n}} \\
 & & \qquad
    +\E\left(\prod_{j=1}^{r_n-3} e^{tY_{2j-1,n}}\right)\left(\E e^{tY_{1,n}}\right)^2.
\end{eqnarray*}
Now, we iterate the procedure above to decompose the mathematical expectation of the product to find
$$
\begin{array}{l}
\displaystyle\E\left(\prod_{j=1}^{r_n} e^{tY_{2j-1,n}}\right) \\
\displaystyle\qquad=
   \sum_{j=1}^{r_n-1}\left(\E e^{tY_{1,n}}\right)^{j-1} \Cov\left(\prod_{k=1}^{r_n-j} e^{tY_{2k-1,n}},e^{tY_{2(r_n-j)+1,n}}\right)
   +\left(\E e^{tY_{1,n}}\right)^{r_n}.
\end{array}
$$
The L-weak dependence 
of the variables implies that
\begin{equation}\label{eq:quasi}
\begin{array}{l}
\displaystyle\Abs{\Cov\left(\prod_{k=1}^{r_n-j} e^{tY_{2k-1,n}},e^{tY_{2(r_n-j)+1,n}}\right)} \\
\displaystyle\qquad\qquad\leq
           t^2e^{tcp_n(r_n-j+1)}
           \sum_{k=1}^{r_n-j}
           \sum_{\ell=2(k-2)p_n+1}^{(2k-1)p_n}\sum_{\ell^\prime=2(r_n-j)p_n+1}^{(2(r_n-j)+1)p_n}
           \gamma_{\ell^\prime-\ell}.
\end{array}
\end{equation}
The summation above is similar to the one treated in the course of proof of Lemma~3.1 in \cite{IR98}. Adapting their arguments, one easily finds that
\begin{eqnarray*}
\lefteqn{\sum_{\ell=2(k-2)p_n+1}^{(2k-1)p_n}\sum_{\ell^\prime=2(r_n-j)p_n+1}^{(2(r_n-j)+1)p_n}
           \gamma_{\ell^\prime-\ell}} \\
 & &\qquad=\sum_{\ell=0}^{p_n-1}(p_n-\ell)\gamma_{2kp_n+\ell}+\sum_{\ell=1}^{p_n-1}(p_n-\ell)\gamma_{2kp_n-\ell} \\
 & &\qquad\leq p_n\sum_{\ell=(2k-1)p_n+1}^{(2k+1)p_n-1}\gamma_\ell,
\end{eqnarray*}
thus,
$$
\sum_{k=1}^{r_n-j}
\sum_{\ell=2(k-2)p_n+1}^{(2k-1)p_n}\sum_{\ell^\prime=2(r_n-j)p_n+1}^{(2(r_n-j)+1)p_n}
\gamma_{\ell^\prime-\ell}
\leq\sum_{k=1}^{r_n-j}p_n\sum_{\ell=(2k-1)p_n+1}^{(2k+1)p_n-1}\gamma_\ell\leq p_nv(p_n).
$$
%
Plug this into (\ref{eq:quasi}) and use the inequality proved in Lemma~\ref{lem2} to obtain upper bounds for $\left(\E e^{tY_{1,n}}\right)^{j-1}$ and $\left(\E e^{tY_{1,n}}\right)^{r_n}$. Finally, remember that $2p_nr_n\leq n$ to conclude the proof.
\end{proof}


\begin{lemma}
\label{lem3}
Assume the conditions of Lemma~\ref{lem2+1} are satisfied. Then, for each fixed $x$ and $n$ large enough, there exists a constant $c_1>0$ such that,
\begin{equation}\label{eq:znod}
\prob{Z_{n,od}>x}\leq\left(\frac{c_1x^2}{4\sigma^4n^2d_n^2}e^{\frac{cx}{4\sigma^2 d_n}}p_n v(p_n)+1\right)\exp\left(-\frac{x^2}{4\sigma^2 n d_n}\right).
\end{equation}
\end{lemma}
\begin{proof}
Using Markov's inequality and taking into account (\ref{eq:lem2}), it follows that
\begin{equation}\label{eq:znod1}
\renewcommand{\arraystretch}{2}
\begin{array}{rcl}
\displaystyle\prob{Z_{n,od}>x} 
& \leq &\displaystyle
t^2e^{\frac{tcn}{2}}p_n v(p_n)e^{-tx}\sum_{j=0}^{r_n-2}\exp\left(jtp_n(2t\sigma^2 d_n-c)\right) \\
 &  &\displaystyle\qquad+\exp\left(t^2\sigma^2 nd_n-tx\right).
\end{array}
\renewcommand{\arraystretch}{1}
\end{equation}
Minimizing the exponent on the second term above leads to the choice $t=\frac{x}{2\sigma^2 n d_n}$, which implies that
$$
t^2\sigma^2 nd_n-tx=-\frac{x^2}{4\sigma^2 n d_n}.
$$
We still have to control the summation on the first term. For this purpose, remark that for the choice of $t$ as above,
$2t\sigma^2 d_n-c=\frac{x}{n}-c$. Thus, as $x$ is fixed, for $n$ large enough $2t\sigma^2 d_n-c<0$, so the series corresponding to this sum is convergent. Finally, remark that, again for the choice made for $t$, we have $tx=\frac{x^2}{2\sigma^2 nd_n}$, so
$e^{-tx}\leq c^\prime\exp\left(-\frac{x^2}{4\sigma^2 n d_n}\right)$, and the proof is concluded.
\end{proof}


\section{Strong laws and convergence rates}
\label{sec:rates}
With the tools proved in the previous section, we may now find conditions for the \slln\ and characterize its convergence rate. The first subsection will deal with bounded random variables, using directly the inequalities of Section~\ref{ineq}, while on the second subsection we will extend these results to arbitrary (unbounded) L-weakly dependent variables by using a truncation technique.

\subsection{The case of bounded variables}
\begin{theorem}
\label{teo1}
Assume that the sequence $X_n$, $n\geq1$, is stationary and L-weakly dependent, 
there exists some $c>0$ such that for every $n\geq 1$, $\Abs{X_n}\leq c$ almost surely and that (\ref{eq:sigma}) holds. Assume that the generalized Cox-Grimmett coefficients (\ref{CG}) satisfy $v(n)=O(\rho^n)$, for some $\rho\in(0,1)$. Then, 
$\frac{1}{n}Z_{n,od}\longrightarrow 0$ almost surely.
\end{theorem}
\begin{proof}
We will bound $\prob{Z_{n,od}>n\varepsilon}$ where, without loss of generality, we choose $0<\varepsilon<c$. Applying (\ref{eq:znod}) with $x=n\varepsilon$, we find the upper bound
\begin{equation}
\label{eq:znod2}
\prob{Z_{n,od}>n\varepsilon}\leq
  \left(\frac{c_1\varepsilon^2}{4\sigma^4d_n^2}e^{\frac{c n\varepsilon}{4\sigma^2 d_n}}p_nv(p_n)+1 \right)\exp\left(-\frac{n\varepsilon^2}{4\sigma^2 d_n}\right).
\end{equation}
We have now to verify the fulfilment of the conditions of Lemma~\ref{lem2} and the convergence of the series appearing in the proof of Lemma~\ref{lem3}. Indeed, when proving this lemma, we verified the assumptions of Lemma~\ref{lem2} with $x$ fixed, while the present choice considers $x$ growing with $n$. First, we need to verify that $t=\frac{\varepsilon}{2\sigma^2 d_n}\leq\frac{d_n-1}{d_n}\frac{1}{cp_n}$, as required to use Lemma~\ref{lem2}. This inequality is equivalent to
\begin{equation}
\label{eps}
\varepsilon\leq
2\frac{\sigma^2}{c}\frac{(d_n-1)}{p_n}.
\end{equation}
Thus, we need to choose the sequences such that $\frac{d_n}{p_n}$ is bounded away from 0. Secondly, for the control of the series appearing in (\ref{eq:znod1}), taking into account that $x=n\varepsilon$, we have that $2t\sigma^2 d_n-c=\varepsilon-c$, so it is bounded away from 0. 
Let us now look at the term inside the large parenthesis in (\ref{eq:znod2}). The growth rate of this term is dominated by the exponential factors, $e^{\frac{c n\varepsilon}{4\sigma^2 d_n}}v(p_n)$, as the remaining terms have polynomial behavior. 
Taking now into account the choice for $t$, it follows easily that $e^{\frac{c n\varepsilon}{4\sigma^2 d_n}}v(p_n)$.
This term is bounded as long as $\frac{p_nd_n}{n}$ remains bounded away from 0.
Hence, assuming the previous conditions on the choices for the sequences $p_n$ and $d_n$, it follows that there exists a constant $C>0$ such that
\begin{equation}
\label{rem1}
\prob{\Abs{Z_{n,od}}>n\varepsilon}\leq C\exp\left(-\frac{n\varepsilon^2}{4\sigma^2 d_n}\right).
\end{equation}
Finally, given $\varepsilon\in(0,c)$, choose $d_n=\frac{n\varepsilon^2}{4\sigma^2\alpha\log n}$, for some $\alpha>1$. It is easily verified that a choice of $p_n=n^\theta$, for some $\theta\in(0,1)$, fulfills the assumptions on the sequences. 
Then, the previous inequality rewrites
$$
\prob{Z_{n,od}>n\varepsilon}\leq C\exp\left(-\alpha\log n\right)=\frac{C}{n^\alpha},
$$
which define a convergent series, thus concluding the proof.
\end{proof}


It is obvious that the result just proved also holds if we replace $Z_{n,od}$ by $Z_{n,ev}$, thus we have the almost sure convergence of $\frac{1}{n}S_{n}$. For sake of completeness, we state this result.
\begin{theorem}
Assume that the conditions of Theorem~\ref{teo1} are satisfied. Then, $\frac{1}{n}S_n\longrightarrow 0$ almost surely.
\end{theorem}

\begin{remark}\label{Rmk-Remaining}
Remark that we did not mention the remaining term $R_n$. In fact, this term is negligible, taking into account that, as $\frac{2cp_{n}}{n}\longrightarrow0$, 
\begin{eqnarray*}
\prob{\Abs{\frac{S_{n}}{n}}\geq\varepsilon}& \leq & \prob{\Abs{\frac{Z_{n,od}}{n}}+\Abs{\frac{Z_{n,ev}}{n}}+\Abs{\frac{R_{n}}{n}}\geq\varepsilon}\\
 & \leq &
\prob{\Abs{\frac{Z_{n,od}}{n}}+\Abs{\frac{Z_{n,ev}}{n}}+\frac{2cp_{n}}{n}\geq\varepsilon}\\
 & \leq &
\prob{\Abs{\frac{Z_{n,od}}{n}}+\Abs{\frac{Z_{n,ev}}{n}}\geq \frac{\varepsilon}{2}}.
\end{eqnarray*}
\end{remark}

We may further  identify a convergence rate for the almost sure convergence above.
\begin{theorem}
\label{teo2}
Assume that the conditions of Theorem~\ref{teo1} are satisfied. Then, $\frac{1}{n}Z_{n,od}\longrightarrow 0$ almost surely with convergence rate $\frac{\log n}{n^{1/2-\delta}}$, where $\delta>0$ is arbitrarily small.
\end{theorem}
%
%
\begin{proof} 
We follow the proof of Theorem~\ref{teo1}, allowing now $\varepsilon$ to depend on $n$, that is, considering $\varepsilon_n$ such that
$$
\varepsilon_n^2=\frac{4\sigma^2\alpha d_n\log n}{n},\qquad\alpha>1.
$$
We need to verify that the condition on $t$ in Lemma~\ref{lem2} is satisfied for an appropriate choice of the sequences $p_n$ and $d_n$, that is, that it holds
$t=\frac{\varepsilon_n}{2\sigma^2 d_n}\leq\frac{d_n-1}{d_n}\frac{1}{cp_n}$. Note that once this is checked, the final arguments of the proof of Theorem~\ref{teo1} follow. So, choose $p_n=n^\theta$, for some $\theta>\frac12$. Remember that we have $t=\frac{\varepsilon_n}{2\sigma^2 d_n}$. We need to choose $d_n\longrightarrow+\infty$ such that
\begin{equation*}
tcp_n=\frac{\alpha^{1/2}c}{\sigma}p_n\left(\frac{d_n\log n}{n}\right)^{1/2}
  \leq d_n-1.
\end{equation*}
As $d_n\longrightarrow+\infty$, $\frac{d_n}{2}<d_n-1$ for $n$ large enough and it suffices to have $tcp_n<\frac{d_n}{2}$, which is equivalent to
$$
\frac{2\alpha^{1/2}cn^\theta(\log n)^{1/2}}{\sigma n^{1/2}}\leq d_n^{1/2}.
$$
This leads to the choice $d_n=O(n^{2\theta-1}\log n)$. The analysis of the exponential terms follows analogously as in the proof of Theorem~\ref{teo1}. Indeed, taking into account the choices made for $t$, $\varepsilon_n$ and $d_n$,
\begin{eqnarray*}
\lefteqn{e^{\frac{tcn}{2}}v(p_n)=O\left(\exp\left(\frac{c n\varepsilon_n}{4\sigma^2 d_n}+n^\theta\log\rho\right)\right)} \\
 & &=O\left(\exp\left(\frac{\alpha^{1/2}c}{2\sigma}\left(\frac{n\log n}{d_n}\right)^{1/2}
          +n^\theta\log\rho\right)\right)\\
 & &=O\left(\exp\left(n^{1-\theta}+n^\theta\log\rho\right)\right),
\end{eqnarray*}
which is bounded as $\theta\in\left(\frac12,1\right)$ and $\rho\in(0,1)$. The convergence rate that follows from the above construction is then of order $\varepsilon_n=O\left(\frac{\log n}{n^{1-\theta}}\right)$. To conclude the proof, just rewrite $\theta=\frac12+\delta$.
\end{proof}

The previous result was proved for $\frac {1}{n}Z_{n,od}$ for convenience of the exposition. An analogous version obviously holds for $\frac {1}{n}Z_{n,ev}$, thus implying the same result for $\frac{1}{n}S_n$ by using the same argument as stated in Remark~(\ref{Rmk-Remaining}). Again, for sake of completeness, we state the final result.
\begin{theorem}
\label{teo2.1}
Assume that the conditions of Theorem~\ref{teo1} are satisfied. Then, $\frac{1}{n}S_{n}\longrightarrow 0$ almost surely with convergence rate $\frac{\log n}{n^{1/2-\delta}}$, where $\delta>0$ is arbitrarily small.
\end{theorem}

\subsection{General random variables}
\label{sub:rates}
We now want to drop the boundedness assumption. To extend the results just proved, we will use a truncation technique together with a control on the tails of the distributions. Define, for a given fixed $c>0$, the nondecreasing function $g_c(x)=\max(\min(x,c),-c)$, performing a truncation at level $c$. Remark that, for every $c>0$, $g_c$ is Lipschtizian with $\norm{g_c}=1$. Choose some sequence $c_n\longrightarrow+\infty$, to be made precise later, and define, for $j,n\geq 1$, the random variables
$$
X_{1,j,n}=g_{c_n}(X_j),\qquad X_{2,j,n}=X_j-X_{1,j,n},
$$
and the partial summations
$$
S_{1,n}=\sum_{j=1}^n(X_{1,j,n}-\E X_{1,j,n}),\qquad S_{2,n}=\sum_{j=1}^n(X_{2,j,n}-\E X_{2,j,n}).
$$
%
\begin{theorem}
\label{teo1.1}
Assume that the L-weakly dependent sequence $X_n$, $n\geq1$, is stationary, (\ref{eq:sigma}) holds, and the generalized Cox-Grimmett coefficients (\ref{CG}) satisfy $v(n)=O(\rho^n)$, for some $\rho\in(0,1)$. Assume further that,
\begin{equation}
\label{eq:lapl}
\exists\tau>3,U>0:\,\sup_{\Abs{t}\leq\tau}\E e^{t\Abs{X}}\leq U.
\end{equation}
Then, $\frac{1}{n}S_n\longrightarrow 0$ almost surely with convergence rate $\frac{(\log n)^{3/2}}{n^{1/2-\delta}}$, where $\delta>0$ is arbitrarily small.
\end{theorem}
\begin{proof}
It is obvious that $\prob{\Abs{S_n}>2n\varepsilon}\leq\prob{\Abs{S_{1,n}}>n\varepsilon}+\prob{\Abs{S_{2,n}}>n\varepsilon}$. 
As Theorem~\ref{teo1} applies it follows, taking into account (\ref{rem1}), that,
$$
\prob{\Abs{S_{1,n}}>n\varepsilon_n}\leq 2C\exp\left(-\frac{n\varepsilon_n^2}{4\sigma^2 d_n}\right).
$$
As in the proof of Theorem~\ref{teo2} choose $\varepsilon_n^2=\frac{4\sigma^2\alpha d_n\log n}{n}$, for some $\alpha>1$. This means that $\prob{\Abs{S_{1,n}}>n\varepsilon_n}\leq 2C n^{-\alpha}$, thus defining a convergent series. As before, choose $p_n=n^\theta$, for some $\theta\in\left(\frac12,1\right)$. As in the proof of Theorem~\ref{teo2}, we need to verify that the assumptions of Lemma~\ref{lem2} are satisfied. Taking into account the bounding value for the truncated variables, the assumption of Lemma~\ref{lem2} is now written as $t=\frac{\varepsilon_n}{2\sigma^2 d_n}\leq\frac{d_n-1}{d_n}\frac{1}{c_n p_n}$, which is equivalent to
$$
c_np_n\varepsilon_n=\frac{\alpha^{1/2}}{\sigma}\left(\frac{d_n\log n}{n}\right)^{1/2}c_np_n
  \leq d_n-1\leq d_n.
$$
Therefore, Lemma~\ref{lem2} is applicable if we choose
$$
d_n^{1/2}\geq\frac{\alpha^{1/2}}{\sigma}\frac{(\log n)^{1/2}}{n^{1/2}}c_np_n. 
$$
Using now the choice for $p_n$, this means we may choose $d_n=\frac{\alpha}{\sigma^2}n^{2\theta-1}c_n^2\log n$, thus obtaining
$$
\varepsilon_n^2=4\alpha^2n^{2\theta-2}c_n^2\log n.
$$
We need now to control $\prob{\Abs{S_{2,n}}>n\varepsilon_n}$. Note first that, taking into account the stationarity,
$$
\prob{{\Abs{S_{2,n}}>n\varepsilon_n}}\leq n\prob{\Abs{X_{2,1,n}-\E X_{2,1,n}}>\varepsilon_{n}}
\leq\frac{n}{\varepsilon_{n}^{2}}\E X_{2,1,n}^2.
$$
Denoting $\bar{F}(x)=\prob{\Abs{X_1}>x}$, we have that
$$
\E X_{2,1,n}^2=-\int_{(c_n,+\infty)} (x-c_n)^2\,\bar{F}(dx)=\int_{c_n}^{+\infty}2(x-c_n)\bar{F}(x)\,dx.
$$
Now, using Markov's inequality, it follows that $\bar{F}(x)\leq e^{-tx}\E e^{t\Abs{X_1}}\leq Ue^{-tx}$, if $t\in(0,\tau)$. Thus, for $t\in(0,\tau)$, by integrating the expression above it follows that
$$
\E X_{2,1,n}^2\leq \frac{2U}{t^2}e^{-tc_n},
$$
so finally,
$$
\prob{\Abs{S_{2,n}}>n\varepsilon_n}\leq\frac{2nU}{t^2\varepsilon_n^2}e^{-tc_n}.
$$
If we now choose $c_n=\log n$ and $t=\alpha+2(1-\theta)$, this upper bound behaves like $n^{-\alpha}$, as the upper bound for $\prob{\Abs{S_{2,n}}>n\varepsilon_n}$. Finally, plug these choices into the expression of $\varepsilon_n$ to explicitly identify the convergence rate, finding
$$
\varepsilon_n=4\alpha^2 \frac{(\log n)^{3/2}}{n^{1-\theta}},
$$
and write $\theta=\frac12+\delta$.
\end{proof}

Note that the convergence rate proved in Theorem~\ref{teo1.1} is close to the optimal convergence rate for the \slln\ for associated random variables which is of order
$\frac{(\log n)^{1/2}(\log\!\log n)^{\eta/2}}{n^{1/2}}$
for arbitrarily small $\eta>0$, as proved by Yang, Su and Yu~\cite{YSY08}.

\section{A Central Limit Theorem}
We now look at the convergence in distribution of sums of L-weakly dependent 
variables, extending a \clt\ (CLT) for associated random variables by Newman~\cite{New80,New84} to the L-weak 
dependence structure. The proof of Newman's result (see Theorem~2 in \cite{New80} or Theorem~12 in \cite{New84}) relies on an inequality for characteristic functions, the Newman inequality for characteristic functions (Theorem~1 in Newman~\cite{New80} or Theorem~10 in Newman~\cite{New84}) that controls the approximation between the joint distribution and the product of the marginal distributions. So, we start by proving a version of this inequality for the present dependence structure.
\begin{theorem}
\label{teo:new}
(Newman's inequality for L-weakly dependent 
random variables) Let $X_1,X_2,\ldots,X_n$ be L-weakly dependent 
random variables. Then, for every $t\in\mathbb{R}$, we have
\begin{equation}
\label{eq:new}
\Abs{\E\left(\prod_{j=1}^{n}e^{itX_j}\right) -\prod_{j=1}^{n}\E\left(e^{itX_j}\right) }
\leq 
4t^2\sum_{j=1}^{n-1}(n-j)\gamma_j.
\end{equation}
\end{theorem}
\begin{proof}
We start by adding and subtracting the appropriate terms to the left side of (\ref{eq:new}) to find,
\begin{eqnarray*}
\lefteqn{\Abs{\E\left(\prod_{j=1}^{n} e^{itX_j}\right) -\prod_{j=1}^{n}\E\left(e^{itX_j}\right) } } \\
 & &\leq\Abs{\E\left(\prod_{j=1}^n e^{itX_j} \right) -
       \E\left(e^{itX_{n}}\right)\E\left( \prod_{j=1}^{n-1} e^{itX_j} \right) } \\
 & &\qquad+\Abs{\E\left(e^{itX_n}\right) \E\left( \prod_{j=1}^{n-1} e^{itX_j}\right)
          -\prod_{j=1}^n\E\left( e^{itX_j}\right)} \\
 & & \leq \Abs{\Cov\left( \prod_{j=1}^{n-1} e^{itX_j},e^{itX_n}\right)}
       +\Abs{\E\left(\prod_{j=1}^{n-1} e^{itX_j}\right)-\prod_{j=1}^{n-1}\E\left(e^{itX_j}\right)}.
\end{eqnarray*}
Iterating now this procedure, we find that
$$
\Abs{\E\left(\prod_{j=1}^{n}e^{itX_j}\right) -\prod_{j=1}^{n}\E\left(e^{itX_j}\right)}
\leq\sum_{m=2}^{n}\Abs{\Cov\left(\prod_{j=1}^{m-1}e^{itX_j},e^{itX_m}\right)}.
$$
To bound the covariance terms above, expand this covariance using the trigonometric representation of the complex exponential to find four terms involving cosine or sinus functions. Now, for example,
$$
\Abs{\Cov\left( \cos \left(t\sum_{j=1}^{m-1}X_j\right),\cos \left( tX_m\right) \right)}\leq t^2\sum_{j=1}^{m-1}\gamma_{m-j},
$$
taking into account that $\norm{\cos(tx)}=t$ and using the L-weak dependence 
of the sequence $X_i$ of random variables. Obviously, the same upper bound applies to the remaining terms, so we finally have
$$
\Abs{\E\left(\prod_{j=1}^{n} e^{itX_j}\right) -\prod_{j=1}^{n}\E\left(e^{itX_j}\right) }
\leq 4t^2\sum_{m=2}^{n}\sum_{j=1}^{m-1}\gamma_{m-j}=4t^2\sum_{j=1}^{n-1} (n-j)\gamma_j.
$$
\end{proof}
Newman's inequality is the main tool for proving a \clt\ for associated random variables (see, for example, Theorem~4.1 in Oliveira~\cite{Oli12}). So, having extended Newman's inequality to L-weakly dependent 
variables, we immediately may state the corresponding CLT. The arguments for the proof are similar to those of Theorem~5 in Newman~\cite{New80}, except on what regards the control of the approximation to independence.
\begin{theorem}
\label{teo:clt}
Let the sequence $X_n$, $n\geq1$, of random variables be centered, L-weakly dependent, strictly stationary and square integrable.
Then, $\frac{1}{\sqrt{n}}S_{n}$ converges in distribution to a centered normal random variable with variance $\sigma^2$.
\end{theorem}
\begin{proof}
The proof is based on a decomposition $S_n$ similar to (\ref{decomp1}), into the sum of blocks of size $p\in\mathbb{N}$, now being fixed, and using (\ref{eq:new}). So, given $p\in\mathbb{N}$, put $m=\lfloor\frac{n}{p}\rfloor$, and redefine the blocks
$$
Y_{j,p}=\sum_{k=(j-1) p+1}^{jp}X_k,\: j=1,\ldots,m,\qquad \text{and}\qquad Y_{m+1,p}=\sum_{k=mp+1}^{n}X_k.
$$
Let $\varphi_n( t)$ represent the characteristic function of $\frac{1}{\sqrt{n}}S_{n}$. We will establish that $\Abs{\varphi_n(t) -e^{-t^{2}\sigma^{2}/2}} \longrightarrow 0$. Let us start by writing
\begin{equation}
\label{eq:dec}
\renewcommand{\arraystretch}{2}
\begin{array}{rcl}
\displaystyle\Abs{\varphi _{n}\left( t\right) -e^{-\frac{t^{2}\sigma ^{2}}{2}}}
& \leq &
\displaystyle\Abs{\varphi_{n}(t) -\varphi_{mp}(t)} +\Abs{\varphi_{mp}(t)-\varphi_p^m (t)} \\
 & &
\displaystyle+\Abs{\varphi_p^m(t) - e^{-\frac{t^{2}\sigma _{p}^{2}}{2}}}
    +\Abs{e^{-\frac{t^{2}\sigma _{p}^{2}}{2}} - e^{-\frac{t^{2}\sigma ^{2}}{2}}},
\end{array}
\renewcommand{\arraystretch}{1}
\end{equation}
where $\sigma_p^2=\frac{1}{\sqrt{p}}\Var(S_{p})$, and prove that each term of the right hand side goes to zero. Let $p$ be fixed for the time being. As what concerns the first term of the upper bound in (\ref{eq:dec}), we have, using Cauchy's inequality,
\begin{eqnarray*}
\lefteqn{\Abs{\varphi_n(t) -\varphi_{mp}(t)} \leq \E\Abs{\exp \left(\frac{it}{\sqrt{n}}S_{n}
\right) -\exp \left(\frac{it}{\sqrt{mp}}S_{mp}\right) }} \\
 & &\leq \Abs{t}\E\Abs{\frac{S_{n}}{\sqrt{n}}-\frac{S_{mp}}{\sqrt{mp}}}
 \leq \Abs{t}\left( \E\left( \frac{S_{n}}{\sqrt{n}}-\frac{S_{mp}}{\sqrt{mp}}\right)^2\right) ^{1/2} \\
 & & \leq \Abs{t}\left( \frac{1}{\sqrt{mp}}-\frac{1}{\sqrt{n}}\right) \left(\E S_{mp}^{2}\right)^{1/2} +\frac{\Abs{t}}{\sqrt{n}}\left(\E Y_{m+1,p}^{2}\right)^{1/2}.
\end{eqnarray*}
It follows from the stationarity of the sequence of random variables $X_i$ that, for $m$ large enough, $\E S_{mp}^2\leq 2\sigma^2mp$ and $\E Y_{m+1,p}^2\leq 2\sigma^2(n-mp)<2\sigma^2 p$. Thus, as $n\longrightarrow+\infty$, which implies that $m\longrightarrow+\infty$, it follows
$$
\Abs{\varphi_n(t) -\varphi_{mp}(t)}\leq \sqrt{2}\Abs{t}\sigma\left(1-\frac{\sqrt{mp}}{\sqrt{n}}+\frac{1}{\sqrt{m}}\right)\longrightarrow 0.
$$
The second term in (\ref{eq:dec}) represents the difference between the joint distribution of the blocks and what we would find if they were independent. To control this term, define $W_{j,p}=\frac{1}{\sqrt{p}}Y_{j,p}$. Taking into account the stationarity of the sequence $X_i$, the characteristic function of $W_{j,p}$ is $\varphi_p(t)$. As the variables $W_{j,p}$ are transformations of $X_{(j-1)p+1},\ldots,X_{jp}$, 
it follows from the definition of L-weak dependence, representing the exponential with the trigonometric functions as done for the proof of Theorem~\ref{teo:new}, that
\begin{equation}
\label{decomp_gamma}
\begin{array}{l}
\displaystyle\Abs{\varphi_{mp}(t) -\varphi_p^m(t)} \\
\displaystyle\qquad
   = \Abs{\E\left(\exp \left( \frac{it}{\sqrt{m}}\sum_{k=1}^m W_{k,p}\right) \right) -\prod_{k=1}^m\E\exp\left( \frac{it}{\sqrt{m}}W_{k,p}\right)} \\
\displaystyle\qquad \leq
  \frac{4t^{2}}{mp}\sum_{\ell=2}^{m-1}\sum_{j=1}^{(\ell-1)p}\sum_{j^\prime=(\ell-1)p+1}^{\ell p}\gamma_{j^\prime-j} \\
\displaystyle\qquad =\frac{2t^2}{mp}\left(\sum_{j,j^\prime=1}^{mp}\gamma_{\Abs{j^\prime-j}}
     -m\sum_{j,j^\prime=1}^p\gamma_{\Abs{j^\prime-j}}\right).
\end{array}
\end{equation}
It is easy to verify that
\begin{equation}
\label{def:D}
  \frac{1}{mp}\sum_{j,j^\prime=1}^{mp}\gamma_{\Abs{j^\prime-j}}
=\sum_{j=1}^{mp-1}(1-\frac{j}{mp})\gamma_j\longrightarrow
D=\sum_{\ell=1}^\infty\gamma_\ell<\infty,
\end{equation}
which gives us,
$$
\limsup_{m\rightarrow+\infty}\Abs{\varphi_{mp}(t) -\varphi_p^m(t)}\leq 2t^2\left(D-\frac{1}{p}\sum_{j,j^\prime=1}^p\gamma_{\Abs{j^\prime-j}}\right).
$$
%
For the third term in (\ref{eq:dec}), the classical \clt\ for independent random variables implies that
$$
\lim_{m\rightarrow+\infty}\Abs{\varphi_{p}^{m}(t) - e^{-\frac{t^2\sigma_p^2}{2}}}\longrightarrow 0.
$$
Concerning the last term in (\ref{eq:dec}), we have
$\Abs{e^{-t^2\sigma_p^2/2} -e^{-t^2\sigma^2/2}}\leq \frac{t^2}{2}\Abs{\sigma_p^2-\sigma^2}$.
So, finally we obtain,
$$
\limsup_{n\rightarrow+\infty}\Abs{\varphi_n(t) - e^{-\frac{t^2\sigma^2}{2}}}
    \leq \frac{t^2}{2}\Abs{\sigma_p^2-\sigma^2}
    +2t^2\left(D-\frac{1}{p}\sum_{j,j^\prime=1}^p\gamma_{\Abs{j^\prime-j}}\right).
$$
Note that the left hand side above does not depend on $p$. Allowing now $p\longrightarrow+\infty$ and taking into account that $\lim_{p\rightarrow+\infty}\sigma_p^2=\sigma^2$,
it follows that 
$$
\limsup_{n\rightarrow+\infty}\Abs{\varphi_n(t) - e^{-\frac{t^2\sigma^2}{2}}} =0.
$$
\end{proof}

We now prove a functional version of Theorem~\ref{teo:clt}, giving sufficient conditions for the convergence in distribution of the partial sums process:
\begin{equation}
\label{eq:partsum}
\xi_n(t)=\frac{1}{\sqrt{n}}\sum_{j=1}^{\lfloor nt\rfloor}X_j,\qquad 0\leq t\leq 1.
\end{equation}
\begin{theorem}
Let  the sequence of random variables $X_n$, $n\geq1$, be centered, L-weakly dependent, strictly stationary satisfying $\E\Abs{X_1}^{4+\delta}<\infty$, for some $\delta>0$, and (\ref{eq:sigma}). If the L-weak dependence coefficients $\gamma_k$, $k\geq 1$, are decreasing such that $\gamma_k=O( k^{-2-8/\delta})$, then $\xi_n(t)$, $n\geq1$, converges in distribution to $\sigma W$, where $W$ is a standard Brownian motion in the Skhorohod space $\mathbb{D}[0,1]$.
\end{theorem}
\begin{proof}
The proof follows the usual arguments to prove the convergence with respect to the Skhorohod topology: prove the convergence of the finite dimensional distributions and the tightness of the sequence. The one dimensional distributions follows directly from Theorem~\ref{teo:clt}. Choose now $k$ points such that $0=u_0\leq u_1<u_2<\cdots<u_k\leq 1$. We shall prove the asymptotic normality of the random vector
$$
H(u_1,\ldots,u_k)=\frac{1}{\sqrt{n}}\left(\xi_n(u_1),\xi_n(u_2)-\xi_n(u_1),\ldots,\xi_n(u_k)-\xi_n(u_{k-1})\right).
$$
Note that, due to the stationarity, it follows again from Theorem~\ref{teo:clt} that each coordinate of $H(u_1,\ldots,u_k)$ is asymptotically centered normal with variance $(u_s-u_{s-1})\sigma^2$, $s=1,\ldots,k$. We now compare the characteristic function of the random vector with the product of the characteristic functions of its margins. Denote on the sequel $T=\max_{s=1,\ldots,k}\Abs{t_s}$. From the definition of L-weak dependence, reasoning as for the decomposition (\ref{decomp_gamma}), taking into account that $\norm{\cos(\sum_j t_jX_j)}=\max_{j=1,\ldots,k}\Abs{t_j}$, it follows that, for every $t_1,\ldots,t_k\in\mathbb{R}$,
$$
\renewcommand{\arraystretch}{3}
\begin{array}{l}
\displaystyle
\left\vert\E\exp\left(\frac{i}{\sqrt{n}}\sum_{s=1}^k t_s\left(\xi_n(u_s)-\xi_n(u_{s-1})\right)\right)\right.  \\
\displaystyle\qquad\qquad\left.
-\prod_{s=1}^k\E\exp\left(\frac{i t_s}{\sqrt{n}}\left(\xi_n(u_s)-\xi_n(u_{s-1})\right)\right)\right\vert \\
\displaystyle\qquad\qquad\qquad\qquad
\leq\frac{4kT^2}{n}\sum_{s=2}^{k-1}
      \sum_{j=1}^{\lfloor nu_{s-1}\rfloor}
          \sum_{j^\prime=\lfloor nu_{s-1}\rfloor+1}^{\lfloor nu_s\rfloor}\gamma_{j^\prime-j} \\
\displaystyle\qquad\qquad\qquad\qquad
=\frac{2T^2}{n}\left(\sum_{j,j^\prime=1}^{\lfloor nu_k\rfloor}\gamma_{\Abs{j^\prime-j}}
-\sum_{s=1}^k\sum_{j,j^\prime=\lfloor nu_{s-1}\rfloor+1}^{\lfloor nu_s\rfloor}\gamma_{\Abs{j^\prime-j}}
\right).
\end{array}
\renewcommand{\arraystretch}{1}
$$
Note that our assumption on the decrease rate of the $\gamma_j$ coefficients implies the convergence of the corresponding series. So, defining $D$ as in (\ref{def:D}), the above expression is easily seen to converge to $2T^2D(u_k-u_1-(u_2-u_1)-\cdots-(u_k-u_{k-1}))=0$, hence the asymptotic normality of $H(u_1,\ldots,u_k)$ follows.

To complete the proof, we still have to prove the tightness. We follow the arguments in the proof of Theorem~5 in Doukhan and Louhichi~\cite{DL99}, thus needing to prove that
\begin{equation}
\label{eq:lou}
\renewcommand{\arraystretch}{1.75}
\begin{array}{l}
\displaystyle\sum_{j=1}^\infty j\Abs{\E(X_1X_{j+1})}<\infty, \\
\displaystyle\Cov(X_iX_j,X_kX_\ell)=O((k-j)^{-2}),\quad 1\leq i\leq j<k\leq\ell.
\end{array}
\renewcommand{\arraystretch}{1}
\end{equation}
As what concerns the first condition, as the variables are centered and taking into account the assumption on the decrease rate of the $\gamma_\ell$ coefficients:
$$
\sum_{j=1}^\infty j\Abs{\E(X_1X_{j+1})}\leq\sum_{j=1}^\infty j\gamma_j<\infty.
$$
Concerning the second condition in (\ref{eq:lou}), write first, for some $c>0$ and for each $k\geq 1$, $V_k=X_k-(g_c(X_k)-\E g_c(X_k))$, using the function $g_c(\cdot)$ introduced in Subsection~\ref{sub:rates}. Using  this representation the covariance $\Cov(X_iX_j,X_kX_\ell)$ is written as a sum of terms of the form $\Cov(U_1U_2,U_3U_4)$ where each $U_j$ is either bounded by $2c$ or chosen among $V_i$, $V_j$, $V_k$ or $V_\ell$. If all the $U_j$'s are bounded by $2c$, from the definition of L-weak dependence and the assumption that coefficients are decreasing, it follows that
$$
\Abs{\Cov(U_iU_j,U_kU_\ell)}\leq c^4(\gamma_{k-i}+\gamma_{k-j}+\gamma_{\ell-i}+\gamma_{\ell-k})\leq 4c^4\gamma_{k-j}.
$$
If exactly one of the $U_j$'s is not bounded, say $U_i=Y_i$, we have that, using H\"older inequality followed by Markov inequality,
$$
\Abs{\Cov(Y_iU_j,U_kU_\ell)}\leq 2c^3\E\Abs{Y_i}=2c^3\E(\Abs{X_1}\mathbb{I}_{\Abs{X_1}>c})
\leq 2c^{-\delta}\E\Abs{X_1}^{4+\delta}.
$$
For the remaining terms, we may reason in the same way, always finding an upper bound that, up to multiplication by a constant, is $c^{-\delta}\E\Abs{X_1}^{4+\delta}$. Thus, summing all the terms, we have that $\Cov(X_iX_j,X_kX_\ell)=O(c^{-\delta}+c^4\gamma_{k-j})$. Choose now $c=\gamma_{k-j}^{-1/(4+\delta)}$ to find $\Cov(X_iX_j,X_kX_\ell)=O(\gamma_{k-j}^{\delta/(4+\delta)})=O(k-j)^{-2}$, taking into account the decrease rate for the dependence coefficients. So, the tightness follows, which concludes the proof of the theorem.
\end{proof}

This result complements Theorem~5 in Doukhan and Louhichi~\cite{DL99}. Indeed, these authors proved a similar result, but considering different forms of weak dependence, as expressed by their $\psi$ coefficients which involved the sum of the Lipschitz norms of the transformations instead of the product as we considered in Definition~\ref{def:weak}. It is still possible to prove a result concerning the convergence of the empirical process, again somehow in a similar way as done in Doukhan and Louhichi~\cite{DL99}. For this later result, in \cite{DL99} a different dependence coefficient was considered, so that their result implies directly the corresponding one for L-weakly dependent variables. We state the result here, without proof, for easier reference on asymptotic results on L-weakly dependent variables.

\begin{theorem}
Let $X_n$, $n\geq 1$, be centered, L-weakly dependent, strictly stationary random variables uniformly distributed on $[0,1]$. If the L-weak dependence coefficients $\gamma_k$, $k\geq 1$, are such that $\gamma_k=O(k^{-15/2-\delta})$, for some $\delta>0$, then
$\zeta_n(t)=\sqrt{n}\left(\frac{1}{n}\sum_{j=1}^n\mathbb{I}_{[0,t]}(X_j)-t\right)$, $t\in[0,1]$, $n\geq 1$, converges in distribution in the Skhorohod space $\mathbb{D}[0,1]$ to a centered Gaussian process indexed by $[0,1]$ with covariance operator
$$
\Gamma(s,t)=\sum_{k=1}^{+\infty}\Cov\left(\mathbb{I}_{[0,s]}(X_1),\mathbb{I}_{[0,t]}(X_{k})\right).
$$
\end{theorem}


\begin{thebibliography}{99}

\bibitem{Abe} R.T.\ Abebe and H. Zegeye, Mann and Ishikawa-Type Iterative
Schemes for Approximating Fixed Points of Multi-valued Non-Self Mappings,
Mediterranean Journal of Mathematics, $\left( 2016\right) ,1-16$.

\bibitem{BS07} A. Bulinski and A. Shashkin, Limit theorems for associated random fields and related systems, World Scientific Publishing Co. Pte. Ltd., Hackensack, NJ, 2007.
%
\bibitem{BulSuq01} A. Bulinski and C. Suquet, Normal approximation for quasi-associated random fields, Statist. Probab. Lett. 54 (2001), pp. 215--226.

\bibitem{Ded07} J. Dedecker, P. Doukhan, G. Lang, J.R. Leon, S. Louhichi and C. Prieur,  Weak Dependence: With Examples and Applications, Springer, New York, 2007.

\bibitem{DL99} P. Doukhan and S. Louhichi, A new weak dependence condition and applications to moment inequalities, Stoch. Proc. Appl. 84 (1999), pp. 313--342.

\bibitem{EPW67} J. Esary, F. Proschan and D. Walkup, Association of random variables with applications, Ann. Math. Statist. 38 (1967), pp. 1466--1474.

\bibitem{IR98} D. Ioannides and G. Roussas, Exponential inequality for associated random variables, Statist. Probab. Letters 42 (1998), pp. 423--431.

\bibitem{JP83} K. Joag-Dev and F. Proschan, Negative association of random variables with applications, Ann. Statist. 11 (1983), pp. 286--295.

\bibitem{KN06}R. Kallabis and M. Neumann, An exponential inequality under weak dependence, Bernoulli~12 (2006), pp. 333--350.

\bibitem{Leh66} E. Lehmann, Some concepts of dependence, Ann. Math. Statist. 37 (1966), pp. 1137--1153.

\bibitem{L09} L. Liu, Precise large deviations for dependent random variables with heavy tails, Statist. Probab. Letters 79 (2009), pp. 1290--1298.

\bibitem{New80}C. Newman, Normal fluctuations and the FKG inequalities, Comm. Math Phys. 74 (1980), pp.~119--128.

\bibitem{New84} C. Newman, Asymptotic independence and limit theorems for positively and negatively dependent random variables, In: Y.~Tong (ed.) Inequalities in statistics and probability, vol.~5, pp.~127--140, Inst. Math. Statist., Hayward, CA (1984)

\bibitem{Oli05} P.E. Oliveira, An exponential inequality for associated variables, Statist. Probab. Letters~73 (2005), pp. 189--197.

\bibitem{Oli12} P.E. Oliveira, Asymptotics for Associated Random Variables, Springer, Heidelberg, 2012.

\bibitem{PR12} B.L.S. Prakasa Rao, Associated sequences, demimartingales and nonparametric inference, Birkh\"{a}user/Springer, Basel AG, Basel, 2012.

\bibitem{Rio2013} E. Rio, Inequalities and limit theorems for weakly dependent sequences, (2017, January 30) https://hal.archives-ouvertes.fr/cel-00867106/document, 2013.

\bibitem{Sung07} S. Sung, A note on the exponential inequality for associated random variables, Statist. Probab. Letters 77 (2007), pp. 1730--1736.

\bibitem{WWG13} K. Wang, Y. Wang and Q. Gao, Uniform asymptotics for the finite-time ruin probability of a new dependent risk model with a constant interest rate, Method. Comput. Appl. Probab.~15 (2013), pp. 109--124.

\bibitem{YSY08} S. Yang, C. Su and K. Yu, A general method to the strong law of large numbers and its applications, Statist. Probab. Letters 78 (2008), pp. 794-803.






\end{thebibliography}
\end{document}